\documentclass[12pt]{amsart}
\usepackage[utf8]{inputenc}

\usepackage[margin=1.2in]{geometry}                
\usepackage{amsmath,amsthm,amssymb}
\usepackage{graphicx}
\usepackage{comment}
\usepackage{epstopdf}
\usepackage{tikz}
\usepackage{xcolor}
\usepackage[colorlinks=true, pdfstartview=FitV, linkcolor=blue, citecolor=blue, urlcolor=blue]{hyperref}
\usepackage[nameinlink]{cleveref}
    \crefname{conj}{conjecture}{conjectures}
    \crefname{algocfline}{algorithm}{algorithms}

\usepackage{pgfplots, multicol, float, mathtools, extarrows}
\usepgfplotslibrary{fillbetween}
\definecolor{NordDarkBlack}{HTML}{2E3440}     
\definecolor{NordBlack}{HTML}{3B4252}         
\definecolor{NordMediumBlack}{HTML}{434C5e}   
\definecolor{NordBrightBlack}{HTML}{4C566A}   
\definecolor{NordWhite}{HTML}{E5E9F0}         
\definecolor{NordBrightWhite}{HTML}{ECEFF4}   
\definecolor{NordCyan}{HTML}{8FBCBB}          
\definecolor{NordBrightCyan}{HTML}{88C0D0}    
\definecolor{NordBlue}{HTML}{81A1C1}          
\definecolor{NordBrightBlue}{HTML}{5E81AC}    
\definecolor{NordRed}{HTML}{BF616A}           
\definecolor{NordOrange}{HTML}{D08770}        
\definecolor{NordYellow}{HTML}{EBCB8B}        
\definecolor{NordGreen}{HTML}{A3BE8C}         
\definecolor{NordMagenta}{HTML}{B48EAD}       

\usepackage[linesnumbered,ruled,vlined]{algorithm2e}

\makeatletter  
\@namedef{subjclassname@2020}{%
  \textup{2020} Mathematics Subject Classification}
\makeatother

\newcommand{\N}{{\mathbb N}}
\renewcommand{\P}{{\mathbb P}}
\newcommand{\R}{{\mathbb R}}
\newcommand{\Q}{{\mathbb Q}}
\newcommand{\Z}{{\mathbb Z}}

\newcommand{\ba}{{\mathbf a}}
\newcommand{\bb}{{\mathbf b}}

\newcommand{\be}{{\mathbf e}}

\newcommand{\bt}{{\mathbf t}}
\newcommand{\bx}{{\mathbf x}}
\newcommand{\by}{{\mathbf y}}

\newcommand{\bv}{{\mathbf v}}

\newcommand{\m}{\mathfrak{m}}

\def\Ass{\operatorname{Ass}}

\def\conv{\operatorname{convex\ hull}}

\def\het{\operatorname{ht}}

\def\Max{\operatorname{Max}}

\def\ov{\overline}

\def\Irr{\operatorname{Irr}}

\newcommand{\bight}{\operatorname{big-height}}

\newtheorem{thm}{Theorem}[section]

\newtheorem*{introthm*}{Theorem}
\newtheorem{conj}[thm]{Conjecture}
\newtheorem{cor}[thm]{Corollary}
\newtheorem{lem}[thm]{Lemma}
\newtheorem{prop}[thm]{Proposition}

\theoremstyle{definition}
\newtheorem{defn}[thm]{Definition}

\newtheorem{ex}[thm]{Example}

\theoremstyle{remark}
\newtheorem{rem}[thm]{Remark}
\newtheorem{notation}[thm]{Notation}

\numberwithin{equation}{section}  

\title[Convex bodies and asymptotic invariants]{Convex bodies and asymptotic invariants for powers of monomial ideals}
\author[Polymath 2020, Monomials, Convex Bodies, and Optimization Team]{Jo\~{a}o Camarneiro, Ben Drabkin, Duarte Fragoso, William Frendreiss, Daniel Hoffman, Alexandra Seceleanu, Tingting Tang, Sewon Yang}

\address{Departamento de Matem\'atica, Instituto Superior T\'{e}cnico - Universidade de Lisboa, Av. Rovisco Pais, 1049-001 Lisboa, Portugal}
\email{joao.camarneiro@tecnico.ulisboa.pt}
\address{Information Systems Technology and Design,
Singapore University of Technology and Design, 8 Somapah Road, Building 1, Level 5
Singapore 487372}
\email{benjamin\_drabkin@sutd.edu.sg}
\address{Departamento de Matem\'atica, Instituto Superior T\'{e}cnico - Universidade de Lisboa, Av. Rovisco Pais, 1049-001 Lisboa, Portugal}
\email{duartefragoso@tecnico.ulisboa.pt}
\address{Department of Mathematics, Texas A\&M University, Mailstop 3368, College Station, TX 77843, USA}
\email{wfrendreiss@gmail.com}
\address{Department of Mathematics and Statistics, Washington University in St. Louis, 1 Brookings Drive, St. Louis, MO 63130, USA}
\email{danielhhoffman@gmail.com}
\address{Department of Mathematics, University of Nebraska--Lincoln, 203 Avery Hall, Lincoln, NE 68588, USA}
\email{aseceleanu@unl.edu}
\address{Department of Mathematics and Statistics, San Diego State University, 5500 Campanile Drive, 
San Diego, CA 92182, USA}
\email{ttang2@sdsu.edu}
\address{Department of Mathematics, University of Maryland, William E. Kirwan Hall, 4176 Campus Dr, College Park, MD 20742, USA}
\email{syang132@terpmail.umd.edu}

\begin{document}

\begin{abstract}
Continuing a well established tradition of associating convex bodies to monomial ideals,  we initiate a program to construct asymptotic Newton polyhedra from decompositions of monomial ideals. This is achieved by forming a graded family of ideals based on a given decomposition. We term these graded families powers since they generalize the notions of ordinary and symbolic powers. 
Asymptotic invariants for these graded families are expressed as solutions to linear optimization problems on the respective convex bodies. 
This allows to establish a lower bound on the Waldschmidt constant of a monomial ideal by means of a more easily computable invariant, which we introduce under the name of naive Waldschmidt constant.
 \end{abstract}


\maketitle

\section{Introduction}

This paper concerns invariants of monomial ideals which admit interpretations from a convex geometry perspective. Monomial ideals are ideals $I$ that can be generated by monomials in a polynomial ring $R=K[x_1,\ldots, x_n]$ with coefficients in a field $K$.

There is a well established tradition of associating convex bodies to monomial ideals. The preeminent example in this direction is the Newton polyhedron, which is the convex hull of all the exponent vectors of monomials in $I$. Invariants of monomial ideals can be read from their Newton polyhedron: for example, the Hilbert-Samuel multiplicity of an ideal primary to the homogeneous maximal ideal  can be interpreted as the normalized volume of the complement of its Newton polyhedron. For an introduction to the significance of Newton polyhedra in commutative algebra with emphasis on the role they play in integral closure we recommend \cite[\S 1.4, \S 10.3, \S 11]{SH}.

In this paper we focus our attention on associating convex bodies to decompositions of a monomial ideal as an intersection of monomial ideals. Such a decomposition $I=J_1\cap \cdots \cap J_s$ leads to considering on one hand a graded family of monomial ideals 
\begin{equation}
\label{eq:family}
I_m=J_1^m\cap \cdots \cap J_s^m
\end{equation}
 obtained by intersecting the powers of the components in the original decomposition. On the other hand it leads to considering a convex body 
 \[
 \mathcal{C}=NP(J_1)\cap \cdots \cap NP(J_s)
 \]
  obtained by intersecting the Newton polyhedra of the components in the original decomposition. Our first main result shows that $\mathcal{C}$ can be understood as a limit of the Newton polyhedra for the family of ideals $\{I_m\}_{m\geq 1}$, appropriately scaled. For this reason we term $\mathcal{C}$ the asymptotic Newton polyhedron of the family $\{I_m\}$. 
 
 \begin{introthm*}[\Cref{thm:limitshapeintcldecomp}] 
 If $J_1$, \ldots, $J_s$ are monomial ideals and $I_m=J_1^m\cap \cdots \cap J_s^m$, then there  is an equality of polyhedra
\[\mathcal{C}=NP(J_1)\cap \cdots \cap NP(J_s)=\bigcup_{m\geq 1} \frac{1}{m}NP(I_m).\]
\end{introthm*}
The idea of associating an asymptotic Newton polyhedron to a graded family of monomial ideals has appeared previously in the context of Okounkov bodies attached to a graded linear series \cite{LM, KK}. To our knowledge, asymptotic Newton polyhedra arising from ideal decompositions have not been studied before.

Our work is motivated by the family of symbolic powers of a monomial ideal. Symbolic powers are a topic of sustained interest from a geometric as well as a combinatorial viewpoint. We recommend \cite{dao2017symbolic}, \cite{GSsurvey} for an introduction to this family of ideals and some combinatorial connections. Symbolic powers of monomial ideals fit in the paradigm of the graded families described in \eqref{eq:family} since they arise by intersecting powers of the components in a coarsening of the primary decomposition of the monomial ideal; see \Cref{lem:HHT}. The convex body $\mathcal{C}$ which corresponds to the graded family of symbolic powers is know as the {\em symbolic polyhedron}. It was introduced in \cite{cooper2017symbolic} and utilized in \cite{bocci2016waldschmidt}. In the study of symbolic powers, convex bodies reach beyond the setting of monomial ideals. Indeed, \cite{Mayes} associates a graded family of monomial ideals termed a generic initial system to the symbolic powers of certain ideals in polynomial rings; see also \cite{Walker} for a similar approach. This suggest that our methods can yield future extensions to arbitrary ideals by following this procedure.

A novel family of monomial ideals, termed irreducible powers, is introduced in this paper. They arise from a decomposition of a monomial ideal into irreducible ideals in the manner described in \eqref{eq:family} and coincide with the symbolic powers in some cases of interest, for example for square-free monomial ideals. The advantage to considering the irreducible powers is that for non square-free monomial ideals they give rise to convex bodies, termed {\em irreducible polyhedra}, which are easier to control than the symbolic polyhedra. 
Our second main result captures the symbolic polyhedron between the two other convex bodies discussed above.
 \begin{introthm*}[\Cref{thm:naivecontainment}] 
 For any monomial ideal $I$ the following containments hold between its Newton (NP), symbolic (SP) and irreducible (IP) polyhedra:
\[NP(I)\subseteq SP(I) \subseteq IP(I).\]
\end{introthm*}

We  show that certain asymptotic invariants for graded families of ideals can be read off the respective asymptotic Newton polyhedra by means of linear optimization. These invariants generalize the notion of initial degree of a homogeneous ideal, by which we mean the least degree of a nonzero element of the ideal, to an asymptotic counterpart. For symbolic powers, the resulting asymptotic invariant is known in the literature as the {\em Waldschmidt constant}. It has been investigated in many works, among which we cite \cite{Skoda, Waldschmidt, HaHu, BoH} and specifically for the case of monomial ideals in \cite{cooper2017symbolic, bocci2016waldschmidt}.  In \cite{Nagata59} Nagata established, in different language, that the Waldschmidt constant for any set of $r\geq 9$ very general points in $\P^2$ is $\sqrt{r}$ if $r$ is a perfect square and conjectured this remains true for arbitrary $r\geq 9$. Since then a considerable amount of effort has gone towards providing bounds for Waldschmidt constants in the setting of ideals defining reduced sets of points in projective space; see \cite{EsnaultViehweg, FMXChudnovsky, MTGChudnovsky, ChudnovskyGeneralPoints}. 

In this paper we introduce an analogous invariant, termed {\em naive Waldschmidt constant}, which can be interpreted as the solution of a linear optimization problem on the irreducible polyhedron and which gives an intrinsic lower bound on the Waldschmidt constant. We view the naive Waldschmidt constant as an invariant that is more amenable to computations than the Waldschmidt constant yet it furnishes strong bounds on the latter.
We obtain lower bounds on the naive Waldschmidt constant reminiscent of a Chudnovsky-type inequality conjectured in \cite[Conjecture 6.6]{cooper2017symbolic}. Our results in this direction are summarized by the  inequalities in the following result; the first two inequalities reflect the containments in the previous theorem and the last two inequalities are tight by \Cref{cor:tight}.
\begin{introthm*}[\Cref{LowerBoundThm}] 
Let $I\subseteq K[x_1,\dots,x_n]$ be a monomial ideal with initial degree $\alpha(I)=d$. If $d-1\equiv k\mod(n)$, $0\leq k<n$, then the following inequalities are satisfied by the Waldshchmidt constant  $\widehat{\alpha}(I)$ and the naive Walshchmidt constant  $\widetilde{\alpha}(I)$ 
$$ 
\alpha(I)\geq \widehat{\alpha}(I)
\geq \widetilde{\alpha}(I)
\geq\frac{(n+d-1-k)(2n+d-1-k)}{n(2n+d-1-2k)} \geq \left \lfloor \frac{\alpha(I)+n-1}{n}\right \rfloor.
$$

\end{introthm*}

Our paper is organized as follows: in section 2 we discuss notions of powers arising from decompositions of monomial ideals, in section 3 we associate asymptotic Newton polyhedra to the families introduced previously, in section 4 we define the asymptotic initial degrees for our graded families, we express these invariants by means of linear optimization, and we derive  bounds on their values.

\section{Decompositions of monomial ideals and notions of powers}

Let $\N$ denote the set of nonnegative integers. For all vectors $\ba=(a_1,\ldots,a_n)\in \N^n$ we use the shorthand notation $\bx^\ba:=x_1^{a_1}\cdots x_n^{a_n}$ and thus any monomial ideal is described by a finite set of vectors
$\ba_1,\ldots, \ba_\ell\in \N^n$ as 
$
I=(\bx^{\ba_1},\ldots, \bx^{\ba_n}).
$

An ideal $J$ is called {\em irreducible} if whenever there is a decomposition $J=J_1\cap J_2$, with $J_1, J_2$ ideals, then $J=J_1$ or $J=J_2$. 
An {\em irreducible decomposition} of an ideal $I$ is an expression $I=J_1\cap J_2 \cap \cdots \cap J_s$ where $J_i$ are irreducible ideals for all $1\leq i\leq s$. Such a decomposition is called {\em irredundant} if none of the $J_i$ can be omitted from this expression. Emmy Noether showed in \cite{Noether} that every ideal $I$ in a noetherian ring admits an irredundant irreducible decomposition. Moreover, although the number of components in any irredundant irreducible decomposition of $I$ is the same, the components themselves are in general not unique. 

A monomial ideal $J$ is irreducible if and only if it is generated by pure powers for a subset of the variables, i.e., $J=(x_{i_1}^{a_1}, \cdots, x_{i_t}^{a_t})$. By contrast to arbitrary ideals, monomial ideals have a unique irredundant decomposition into irreducible monomial ideals up to permutation of the components; see \cite[Theorem 3.3.9]{MRSW} for a proof. The ideals appearing in a monomial  irredundant irreducible  decomposition of a monomial ideal $I$ can also be characterized as the smallest irreducible monomial ideals which contain $I$. See \Cref{def:irr} and the considerations following it for this perspective.

Irreducible decompositions are special cases of primary decompositions. Both for monomial and for arbitrary ideals they possess the advantage of being much more easily computable in an algorithmic fashion; see \cite[\S 5.2]{MillerSturmfels} and \cite{FortunaGianniTrager}. In the case of square-free  monomial ideals and more generally for radical ideals, the irredundant irreducible decomposition and the irredundant primary decomposition coincide.

Symbolic powers of ideals arise from the theory of primary decomposition. For an ideal $I$, the symbolic powers retain only the components of the ordinary powers whose radicals are contained in some associated prime of $I$. When $I$ is a radical ideal and $K$ is a field of characteristic 0, the $m$-th symbolic power of $I$ encodes the polynomial functions vanishing to order at least $m$ on the variety cut out by $I$. 

\begin{defn}
\label{def:symbolicpower}
Let $R$ be a noetherian ring and $I$ an ideal in $R$.
The {\em $m$-th symbolic power} of $I$ is the ideal 
$$I^{(m)} = \bigcap_{P \in \Ass(R/I)} I^m R_p \cap R.$$
\end{defn}

Recall that the set of associated primes, denoted $\Ass(R/I)$, of an ideal $I$ in a Noetherian ring is finite. We view it as a poset with respect to containment. A minimal element of this poset is called a minimal prime of $I$ and the non minimal elements are called embedded primes. 

The symbolic powers of monomial ideals admit an alternate description which is even more closely related to their primary decomposition.

\begin{lem}[{\cite[Lemma 3.1]{HHT}}, {\cite[Theorem 3.7]{cooper2017symbolic}}]
\label{lem:HHT}
If $I$ is a monomial ideal with monomial primary decomposition $I=Q_1\cap Q_2\cap \cdots\cap Q_s$, set $\Max(I)$ to denote the set of maximal elements in the poset of associated primes of $I$ and for each $P\in \Max(I)$ denote 
\[Q_{\subseteq P}=\bigcap_{\sqrt{Q_i}\subseteq P} Q_i.\] 

Then the symbolic powers of $I$ can be expressed as follows
\[
I^{(m)}=\bigcap_{P\in \Max(I)}(Q_{\subseteq P})^m.
\]
\end{lem}

\begin{rem}
\label{rem:uniquecombined}
The above Lemma employs the decomposition $I=\bigcap_{P\in \Max(I)}Q_{\subseteq P}$. We will call this a {\em combined primary decomposition} for $I$. The ideals $Q_{\subseteq P}$ are uniquely determined by $I$ and $P$ and are independent of the primary decomposition in the statement of \Cref{lem:HHT}. This follows from the identity  $Q_{\subseteq P}=IR_{P'}\cap R$, where $P'$ is the prime monomial ideal generated by the variables of $R$ that are not in $P$.
\end{rem}

 \begin{ex}
 \label{ex:noemb}
 If $I$ is a monomial ideal with no embedded primes and  $I=Q_1\cap \cdots\cap Q_s$ is a minimal primary decomposition, then the symbolic powers of $I$ are given for all integers $m\geq 1$ by 
 \[ I^{(m)}=Q_1^m\cap Q_2^m \cap \cdots\cap Q_s^m.\] 
 \end{ex}
 
 In this paper we introduce a notion of irreducible powers for monomial ideals, which parallels the behavior in \Cref{ex:noemb}.
 
 \begin{defn}
\label{def:irredpower}
Let $I$ be a monomial ideal with monomial irreducible decomposition given by $I=J_1\cap J_2 \cap \cdots \cap J_s$.
For integers $m\geq 1$, the {\em $m$-th irreducible power} of $I$ is the ideal
$$I^{\{m\}} = J_1^m \cap J_2^m \cap \cdots \cap J_s^m.$$
\end{defn}

It is easy to see that the definition above does not depend on whether the decomposition is irredundant.

 \begin{rem}
 \label{ex:sqfree}
 If $I$ is a square-free monomial ideal then the irredundant irreducible decomposition of $I$ coincides with the combined primary decomposition thus the symbolic powers and irreducible powers of square-free monomial ideals coincide.
 
 More generally if the components in the irredundant irreducible decomposition of $I$ have distinct radicals, then the symbolic powers and irreducible powers  coincide.
 \end{rem}
 
One similarity between the symbolic and irreducible powers is that they both form graded families. A {\em graded family} of ideals $\{I_m\}_{m\in \N}$ is a collection of ideals that satisfies $I_a\cdot I_b\subseteq I_{a+b}$ for all pairs $a,b\in \N$.
 
 \begin{lem}
 The irreducible powers of a monomial ideal form a  graded family, i.e., any nonnegative integers $a,b$ give rise to a containment 
 \[
 I^{\{a\}}\cdot I^{\{b\}} \subseteq I^{\{a+b\}}.
 \]
 \end{lem}
 \begin{proof}
 The containment follows easily from \Cref{def:irredpower}.
 \end{proof}

In many ways, the irreducible powers of monomial ideals resemble closely the symbolic powers of square-free monomial ideals.  A similarity between irreducible powers of monomial ideals and symbolic powers of square-free monomial ideals is that their associated primes are among the associated primes of $I$. This is not the case for symbolic powers of arbitrary ideals; see \Cref{rem:noemb}.
 
 \begin{lem}
 \label{lem:asspowers}
 Let $I$ be a monomial ideal. Then for each integer $m\geq 1$ there are containments $I^m\subseteq I^{(m)}\subseteq I^{\{m\}}$ and $ \Ass(I^{\{m\}})=\Ass(I)$.
 \end{lem}
 \begin{proof}
The containments $I^m\subseteq I^{(m)}\subseteq I^{\{m\}}$ follow from the definition of symbolic powers \Cref{def:symbolicpower} for the former and from \Cref{lem:HHT} for the latter.  In detail, if $I=J_1\cap \cdots\cap J_s$ is an irredundant irreducible decomposition, then for each $1\leq i \leq s$ there exists a prime $P_i\in \Max(I)$ such that $\sqrt{J_i}\subseteq P_i$. Then we see from \Cref{lem:HHT} that $Q_{\subseteq P_i}\subseteq J_i$ and so we deduce
\[
I^{(m)}=\bigcap_{P\in \Max(I)} Q_{\subseteq P}^m \subseteq \bigcap_{i=1}^s Q_{\subseteq P_i}^m  \subseteq \bigcap_{i=1}^s J_i^m =I^{\{m\}}.
\] 
 
Now  let $I=J_1\cap \cdots\cap J_s$ be an irredundant irreducible decomposition with $\mathfrak{p}_i=\sqrt{J_i}$. Since each irreducible ideal $J_i$ is generated by a regular sequence of pure powers of the variables, it follows that $\Ass(J_i^m)=\Ass(J_i)=\{\mathfrak{p}_i\}$ for each $i$ and thus we obtain
\[\Ass(I^{\{m\}})=\Ass(J_1^m\cap \cdots\cap J_s^m)= \{\mathfrak{p}_1,\ldots, \mathfrak{p}_s\}=\Ass(I).\]
 \end{proof}
 
 \begin{rem}
 \label{rem:noemb}
 \Cref{lem:asspowers} reveals that the irreducible powers of monomial ideals enjoy a property that the symbolic powers of monomial ideals which possess embedded primes do not enjoy. Specifically,  it is not in general true that the associated primes of the symbolic powers are restricted to a subset of $\Ass(I)$. For example, the ideal 
 \[ I=(x,y)\cap(x,z)\cap(x,w)\cap(y,z)\cap(y,w)\cap(z,w)\cap(x,y,z,w)^4\] has the property that $\Ass(I^{(2)})$ contains the primes $(x,y,z), (x,y,w),(x,z,w)$ and $(y,z,w)$ in addition to the associated primes of $I$.
 
 \end{rem}

\section{Convex bodies associated to powers of monomial ideals}

\subsection{The symbolic and irreducible polyhedra of a monomial ideal}

In this section we define new convex bodies associated to decompositions of monomial ideals. In our main cases of interest these convex bodies will be polyhedra.  A polyhedron can be defined in two different manners, either as convex hulls of a set of points in Euclidean space or as a finite intersection of half spaces. All polyhedra considered in this section will be unbounded.

 \begin{defn}
For a monomial ideal $I$, the {\em Newton polyhedron} of $I$, denoted $NP(I)$, is the convex hull of the exponent vectors for all the monomials in $I$
\[
NP(I)= \conv  \{\ba\in \N^n \mid \bx^\ba\in I\}.
\]
\end{defn}

One of the useful properties of Newton polyhedra is that they scale linearly upon taking ordinary powers of ideals, namely the following identity holds for all $m\in \N$:
\[NP(I^m)=mNP(I). \] 

The situation becomes more complicated upon considering  Newton polyhedra for the symbolic powers or for the irreducible powers, as taking Newton polyhedra does not commute with intersections of ideals. Specifically, there is always a containment \[NP(J_1\cap \cdots\cap J_s)\subseteq NP(J_1)\cap\cdots \cap NP(J_s),\] but this rarely becomes an equality. However, we shall see that there is an asymptotic sense in which Newton polyhedra can be taken to commute with intersections of ideals. To elaborate on this, we introduce two more convex bodies, one corresponding to each of the notions of symbolic and irreducible powers introduced in the previous section. 

Following \cite[Definition 5.3]{cooper2017symbolic}, which in turn takes inspiration from \Cref{lem:HHT}, we define a symbolic polyhedron associated to a monomial ideal.
\begin{defn}
\label{def:symbolicpolyhedron}
The {\em symbolic polyhedron} of a monomial ideal $I$ with primary decomposition $I=Q_1\cap\cdots\cap Q_s$ is
\[SP(I)=\bigcap_{P\in \Max(I)}NP(Q_{\subseteq P}) \qquad \text{ where }\qquad  Q_{\subseteq P}=\bigcap_{\sqrt{Q_i}\subseteq P} Q_i.\]
This polyhedron does not depend on the choice of primary decomposition by \Cref{rem:uniquecombined}.
\end{defn}

Similarly, with inspiration taken from \Cref{def:irredpower}, we introduce a new convex body termed the irreducible polyhedron.
\begin{defn}
\label{def:irredpolyhedron}
The {\em irreducible polyhedron} of a monomial ideal $I$ with monomial irreducible decomposition $I=J_1\cap\cdots\cap J_s$ is
\[IP(I)=NP(J_1)\cap \cdots \cap NP(J_s).\]
\end{defn}

\begin{ex} 
\label{ex:1}
The figure below shows the Newton, symbolic and irreducible polyhedra for the ideal $I=(x^2, xy, y^2)$ with  irreducible decomposition $I=(x,y^2)\cap(x^2,y)$. The dashed lines help visualize the boundaries of the convex bodies appearing in the definition of the irreducible polyhedron 
\[IP(I)=NP((x,y^2))\cap NP((x^2,y)).\] 
The equality $SP(I)=NP(I)$ follows in this case because $I$ is primary to the maximal ideal $\m=(x,y)$ of $k[x,y]$ and hence the combined primary decomposition consists of a single component, in other words $I=Q_{\subseteq \m}$ and $SP(I)=NP(Q_{\subseteq \m})=NP(I)$.

\begin{figure}[h!]
    \centering
    \pgfplotsset{posQuad/.append style={grid=both, xlabel={$x$}, ylabel={$y$}, line width=2pt, mark size=3pt,draw=NordWhite}}

    \begin{tikzpicture}[scale=0.65]
        \tikzstyle{every node}=[font=\small]
        \pgfplotsset{every axis legend/.append style={legend pos=outer north east,draw=NordWhite}}
        
        \begin{axis}
        [posQuad, xtick={0,...,3}, ytick={0,...,3}, 
        xmin=-0.05, xmax=3, ymin=-0.05, ymax=3]
        
        \addplot[area style, fill=NordBrightBlue, opacity=0.5, line width=0pt, draw = NordBrightBlue] coordinates{(0,3) (0,2) (2,0) (3,0) (3,3)};
        \addlegendentry{$NP(I)=SP(I)$}
        
        \addplot[only marks,mark=*, fill=NordYellow, draw=NordBlack] coordinates{(2,0) (0,2)};
               \end{axis}
        \end{tikzpicture}
   \begin{tikzpicture}[scale=0.65]
        \tikzstyle{every node}=[font=\small]
        \pgfplotsset{every axis legend/.append style={legend pos=outer north east,draw=NordWhite}}

                \begin{axis}
        [posQuad, xtick={0,...,3}, ytick={0,...,3}, 
        xmin=-0.05, xmax=3, ymin=-0.05, ymax=3]
        
      \addplot[area style, fill=NordRed, opacity=0.5, line width=0pt, draw = NordRed] coordinates{(0,3) (0,2) (2/3, 2/3) (2,0) (3,0) (3,3)};
        \addlegendentry{$IP(I)$}
        
\addplot[only marks,mark=*, fill=NordYellow, draw=NordBlack] coordinates{(2,0) (2/3, 2/3) (0,2)};
 \addplot[dashed, NordBlack] coordinates{(2,0) (0,1)};
  \addplot[dashed, NordBlack] coordinates{(1,0) (3,0)};
 \addplot[dashed, NordBlack] coordinates{(0,2) (1,0)};
  \addplot[dashed, NordBlack] coordinates{(0,1) (0,3)};
     \end{axis}
   \end{tikzpicture}
    \caption{The Newton, symbolic and irreducible polyhedra of $I=(x^2, xy, y^2)$}
    \label{fig:coloredboundaries}
\end{figure}
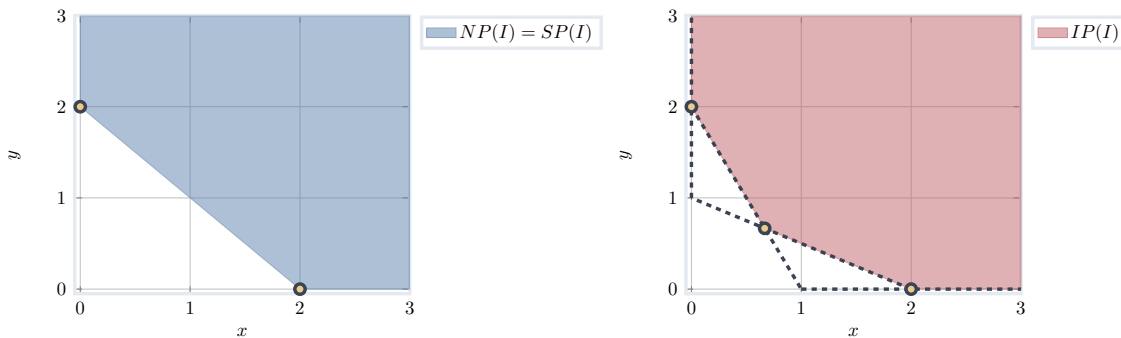
\end{ex}

\begin{rem}
\label{ex:sqfree2}
If $I$ is a square-free monomial ideal or more generally an ideal such that the radicals of the irredundant irreducible components are distinct, then $IP(I)=SP(I)$ by \Cref{ex:sqfree}.
\end{rem}

\begin{ex} 
\label{ex:2}
The figure below shows a partial view of the facets of the Newton, symbolic and irreducible polyhedra for the ideal $I=(xy,xz,yz)$ with irreducible decomposition $I=(x,y)\cap(x,z)\cap(y,z)$. The respective polyhedra are solid bodies located in the positive orthant and having the pictured facets as the outer boundary. In particular that $SP(I)$ has an additional vertex compared to $NP(I)$, which is  located at $(\frac{1}{2},\frac{1}{2},\frac{1}{2})$.
The equality $IP(I)=SP(I)$ follows in this case because $I$ is squarefree; see \Cref{ex:sqfree2}.

\begin{figure}[h!]
   \centering
    \begin{tikzpicture}[scale=0.65]
    \tikzstyle{every node}=[font=\small]
    \pgfplotsset{every axis legend/.append style={legend pos=outer north east,draw=NordWhite}}
        \pgfplotsset{every axis legend/.append style={legend pos=outer north east,draw=NordWhite}}
        
      \begin{axis}[view={-20}{-25}, xtick={0,...,2}, ytick={0,...,2}, ztick={0,...,2},    xmin=0, xmax=2, ymin=0, ymax=2, zmin=0, zmax=2]
        
 \addplot3[surf, faceted color = NordBrightBlue, color = NordBrightBlue, opacity = 0.5, line width = 1pt] coordinates {
     (1,1,0) (0,1,1) 
			 
	 (1,2,0) (0,2,1) 
	};
	 \addplot3[surf, faceted color = NordBrightBlue, color = NordBrightBlue, opacity = 0.5, line width = 1pt] coordinates {
     (1,1,0) (1,0,1) 
			 
	 (2,1,0) (2,0,1) 
	};
	 \addplot3[surf, faceted color = NordBrightBlue, color = NordBrightBlue, opacity = 0.5, line width = 1pt] coordinates {
     (1,0,1) (0,1,1) 
			 
	 (1,0,2) (0,1,2) 
	};
	 \addplot3[surf, faceted color = NordBrightBlue, color = NordBrightBlue, opacity = 0.6, line width = 1pt] coordinates {
     (1,0,1) (1,0,2) 
			 
	 (2,0,1) (2,0,2) 
	};
	 \addplot3[surf, faceted color = NordBrightBlue, color = NordBrightBlue, opacity = 0.6, line width = 1pt] coordinates {
     (1,1,0) (1,2,0) 
			 
	 (2,1,0) (2,2,0) 
	};
	 \addplot3[surf, faceted color = NordBrightBlue, color = NordBrightBlue, opacity = 0.6, line width = 1pt] coordinates {
     (0,1,1) (0,1,2) 
			 
	 (0,2,1) (0,2,2) 
	};
        \addplot3[surf, faceted color = NordBrightBlue, color = NordBrightBlue, opacity = 0.4, line width = 1pt] coordinates {
    (1,1,0) (1,0,1) 
    		 
(1,1,0)    (0,1,1) 	};
        \addlegendentry{$NP(I)$}
     
 \addplot3[only marks,mark=*, fill=NordYellow, draw=NordBlack] coordinates{(1,1,0) (1,0,1) (0,1,1)};
               \end{axis};
        \end{tikzpicture}
         \begin{tikzpicture}[scale=0.65]
    \tikzstyle{every node}=[font=\small]
    \pgfplotsset{every axis legend/.append style={legend pos=outer north east,draw=NordWhite}}
        \pgfplotsset{every axis legend/.append style={legend pos=outer north east,draw=NordWhite}}
        
      \begin{axis}[view={-20}{-25}, xtick={0,...,2}, ytick={0,...,2}, ztick={0,...,2},    xmin=0, xmax=2, ymin=0, ymax=2, zmin=0, zmax=2]
        
    \addplot3[surf, faceted color = NordBrightBlue, color = NordRed, opacity = 0.5, line width = 1pt] coordinates {
     (1,1,0) (0,1,1) 
			 
	 (1,2,0) (0,2,1) 
	};
	 \addplot3[surf, faceted color = NordBrightBlue, color = NordRed, opacity = 0.5, line width = 1pt] coordinates {
     (1,1,0) (1,0,1) 
			 
	 (2,1,0) (2,0,1) 
	};
	 \addplot3[surf, faceted color = NordBrightBlue, color = NordRed, opacity = 0.5, line width = 1pt] coordinates {
     (1,0,1) (0,1,1) 
			 
	 (1,0,2) (0,1,2) 
	};
	 \addplot3[surf, faceted color = NordBrightBlue, color = NordRed, opacity = 0.6, line width = 1pt] coordinates {
     (1,0,1) (1,0,2) 
			 
	 (2,0,1) (2,0,2) 
	};
	 \addplot3[surf, faceted color = NordBrightBlue, color = NordRed, opacity = 0.6, line width = 1pt] coordinates {
     (1,1,0) (1,2,0) 
			 
	 (2,1,0) (2,2,0) 
	};
	 \addplot3[surf, faceted color = NordBrightBlue, color = NordRed, opacity = 0.6, line width = 1pt] coordinates {
     (0,1,1) (0,1,2) 
			 
	 (0,2,1) (0,2,2) 
	};
        \addplot3[surf, faceted color = NordBrightBlue, color = NordRed, opacity = 0.4, line width = 1pt] coordinates {
    (1,1,0)  (0.5, 0.5, 0.5)	
    		 
 (1,1,0)   (1,0,1)
	};
	    \addplot3[surf, faceted color = NordBrightBlue, color = NordRed, opacity = 0.4, line width = 1pt] coordinates {
    (1,0,1)  (0.5, 0.5, 0.5)	
    		 
 (1,0,1)   (0,1,1)
	};
		    \addplot3[surf, faceted color = NordBrightBlue, color = NordRed, opacity = 0.4, line width = 1pt] coordinates {
    (1,1,0)  (0.5, 0.5, 0.5)	
    		 
 (1,1,0)   (0,1,1)
	};
        \addlegendentry{$SP(I)=IP(I)$}
     
 \addplot3[only marks,mark=*, fill=NordYellow, draw=NordBlack] coordinates{(1,1,0) (1,0,1) (0,1,1) (0.5, 0.5, 0.5)};
               \end{axis};
        \end{tikzpicture}

    \caption{The Newton, symbolic and irreducible polyhedra of $I=(xy, xz, yz)$}
    \label{fig:coloredboundaries}
\end{figure}
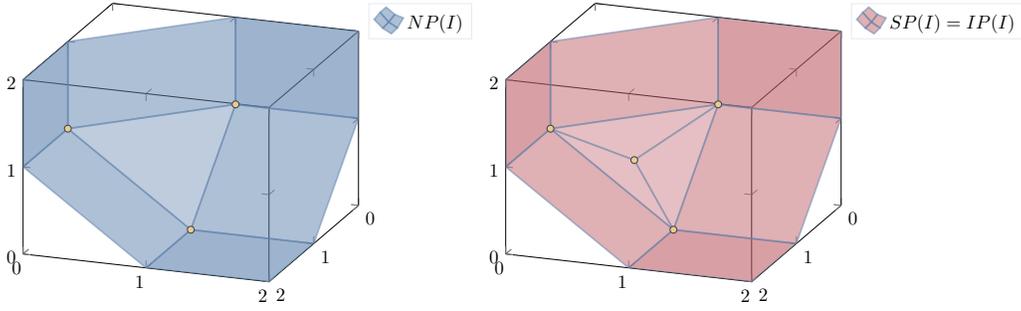
\end{ex}

We supplement the description of the irreducible polyhedron in \Cref{def:irredpolyhedron} by providing equations for hyperplanes supporting the facets of the polyhedron, which we term bounding hyperplanes. We term the linear inequalities describing a polyhedron as an intersection of half spaces its bounding inequalities.

Establishing the bounding inequalities for the symbolic polyhedron of an arbitrary  monomial ideal is generally an infeasible task. However, the analogous task is considerably easier for the irreducible polyhedron.

\begin{lem}
\label{rem:eqNP(Q)}
The bounding inequalities for the irreducible polyhedron of a monomial ideal $I$ are  read off a monomial
irreducible decomposition $I=J_1\cap \cdots J_s$ as follows: if  for each $1\leq i\leq s$ we have $J_i=(x_1^{a_{i1}},\ldots, x_{n}^{a_{in}})$, where $a_{ij}\in \N\cup\{-\infty\}$ with the convention that $x_i^{-\infty}=0$, then $IP(I)$ is the set of points $\by=(y_1,\ldots, y_n)\in \R^n$ which satisfy the system of inequalities \eqref{eq:IP} given below, where we set $\frac{1}{-\infty}=0$
\begin{equation}
\label{eq:IP}
 \begin{cases}
\frac{1}{a_{11}}y_{1}+\cdots +\frac{1}{a_{1n}} y_{n}\geq 1\\
\vdots\\
\frac{1}{a_{s1}}y_{1}+\cdots +\frac{1}{a_{sn}} y_{n}\geq 1\\
y_1,\ldots, y_n\geq 0
\end{cases}.
\end{equation}
\end{lem}
\begin{proof}
For each irreducible component $J_i$ we have that $NP(J_i)$ is the complement within the positive orthant of $\R^n$ of a simplex with vertices given by the origin and the exponent vectors of the minimal monomial generators $x_{1}^{a_{i1}},\ldots, x_{n}^{a_{in}}$ for $J_i$, that is,
\[
NP(J_i)=\begin{cases}
\frac{1}{a_{11}}y_{1}+\cdots +\frac{1}{a_{in}} y_{n}\geq 1\\
y_1,\ldots, y_d\geq 0.
\end{cases}
\]
Equation \eqref{eq:IP} collects together all the inequalities of each $NP(J_i)$ according to \Cref{def:irredpolyhedron}.
\end{proof}

We next give an account of the containments between the three polyhedra discussed above. This is based upon observing that more refined decompositions of an ideal will yield larger polyhedra. We make this precise in the following lemma.

\begin{lem}
\label{lem:refined}
Assume given two collections of monomial ideals $I_1, \ldots,  I_t$ and $J_1, \ldots, J_s$
 such that the latter  refines the former, that is, for each $1\leq j\leq s$ there exists $1\leq i_j \leq t$ such that $I_{i_j}\subseteq J_j$. Then there is a containment of polyhedra 
 \[
NP(I_1)\cap \cdots \cap NP(I_t)\subseteq NP(J_1)\cap \cdots \cap NP(J_s).
 \]
\end{lem}
\begin{proof}
Employing the hypothesis that for each $1\leq j\leq s$ there exists $1\leq i_j \leq t$ such that $I_{i_j}\subseteq J_j$, we deduce that $NP(I_{j_i})\subseteq NP(J_i)$. Thus we obtain the desired containments
\[NP(I_1)\cap \cdots \cap NP(I_t)\subseteq NP(I_{j_i})\cap \cdots \cap NP(I_{j_s})\subseteq NP(J_1)\cap \cdots \cap NP(J_s).
\]
\end{proof}

With this key ingredient in hand we established the containments between the three types of polyhedra considered in this paper.

\begin{thm}
\label{thm:naivecontainment}
Every monomial ideal $I$ satisfies the following containments 
\[NP(I)\subseteq SP(I) \subseteq IP(I).\]
\end{thm}
\begin{proof}
Let \(I=J_1\cap  \cdots\cap J_s\) be a monomial irreducible decomposition, then it is also a primary decomposition. Hence the combined primary decomposition \(I=\bigcap_{P\in\Max(I)} Q_{\subseteq P}\) can be computed using $Q_{\subseteq P}=\bigcap_{\sqrt{J_i}\subseteq P} J_i$ according to \Cref{rem:uniquecombined}. This shows that the irreducible decomposition refines the combined decomposition in the sense that for each $1\leq j \leq s$ there exists $P_j\in \Max(I)$ such that $Q_{\subseteq P_j}\subseteq J_j$. Indeed, this is the case for each $P\in \Max(I)$ such that $\sqrt{J_j}\subseteq P$ and such a prime exists by finiteness of the poset $\Ass(I)$. 

Now we apply \Cref{lem:refined} to obtain the second desired containment
\[
SP(I)=\bigcap_{P\in\Max(I)} NP(Q_{\subseteq P})\subseteq \bigcap_{j=1}^s NP(J_j)=IP(I).
\]
The remaining containment, $NP(I)\subseteq SP(I)$ can be deduced by applying \Cref{lem:refined} to the trivial decomposition $I=I$ and its refinement $I=\bigcap_{P\in \Max(I)} Q_{\subseteq P}$.
\end{proof}

\begin{rem}
The containments $NP(I)\subseteq SP(I) \subseteq IP(I)$ can be strict. See \Cref{ex:1} for an ideal with $NP(I)\subsetneq SP(I)$ and \Cref{ex:2} for an ideal with $SP(I)\subsetneq IP(I)$.
\end{rem}


\subsection{Asymptotic Newton polyhedra for graded families of monomial ideals}

Let $\{I_m\}_{m\geq 1}$ denote a graded family of monomial ideals. By definition, such a family satisfies containments $I_a\cdot I_b\subseteq I_{a+b}$ for each pair $a,b\in \N$. We define a convex body capturing the asymptotics of each such family. This construction bears some resemblance to the Newton-Okounkov bodies of \cite{LM, KK}. A similar construction appears in \cite{Mayes} but for a different family of monomial ideals. 

\begin{defn}
\label{def:limitingshape}
Given a graded family of monomial ideals $\mathcal{I}:=\{I_m\}_{m\geq 1}$, the {\em limiting body} associated to this family is
\[
\mathcal{C}(\mathcal{I})=\bigcup _{m\to \infty} \frac{1}{m} NP(I_m).
\]
If the limiting body is a polyhedron, we call it the {\em asymptotic Newton polyhedron} associated to the family $\mathcal{I}$. For an example of non polyhedral limiting body see \Cref{limnotpolyhedral}.
\end{defn}

\begin{ex}
For the family of ordinary powers $\{I^m\}_{m\in \N}$ of a monomial ideal, the sequence $ \frac{1}{m} NP(I^m)$ is constant, each term being equal to $NP(I)$. Thus the asymptotic Newton polyhedron associated to the family of ordinary powers of $I$ is none other than the Newton polyhedron of $I$ itself.
\end{ex}

\begin{lem}
The limiting body for a graded family of monomial ideals is a convex body. 
\end{lem}
\begin{proof}
Let $\mathcal{I}=\{I_m\}_{m\geq 1}$ be a graded family of monomial ideals. This implies that $(I_m)^k\subseteq I_{mk}$ for all $k\geq 1$ and hence $\frac{1}{m}NP(I_m)\subseteq \frac{1}{mk} NP(I_{mk})$. Now let $\ba, \bb\in \mathcal{C}( \mathcal{I})$ and suppose $\ba\in \frac{1}{a}NP(I_a)$ and  $\bb\in \frac{1}{b}NP(I_b)$. Then by the preceding argument $\ba, \bb$ are points of the same convex body $\frac{1}{ab}NP(I_{ab})$, which is a subset of $\mathcal{C}(\mathcal{I})$. Thus any convex combination of $\ba, \bb$ is also in $\frac{1}{ab}NP(I_{ab})$ and hence in $\mathcal{C}(\mathcal{I})$.
\end{proof}

\begin{rem}
\label{limnotpolyhedral}
Limiting bodies for arbitrary graded families of monomial ideals can fail to be polyhedral. They can also fail to be closed in the Euclidean topology. Consider for example, the family $\mathcal{I}$ of monomial ideals $I_m\subseteq k[x,y]$ such that $x^ay^b\in I_m$ if and only if $ab>m$. Then $\mathcal{C}(\mathcal{I})=\{(a,b) \mid ab> 1, a\geq 0, b\geq 0\}$ is a non-polyhedral convex region in $\R^2$ and is not closed in the Euclidean topology.
\end{rem}

We now consider graded families arising from decompositions into monomial ideals. In this scenario the limiting body is a polyhedron that can be described explicitly.

\begin{thm}
\label{thm:limitshapeintcldecomp}
Let $I$ be a monomial ideal equipped with a decomposition into monomial ideals $I=J_1\cap \cdots \cap J_s$. Consider the graded family $\mathcal{I}=\{I_m\}_{m\geq 1}$ where 
\[ I_m=J_1^m\cap \cdots \cap J_s^m. \]
Then the limiting body associated to this family is a polyhedron termed the asymptotic Newton polyhedron of this family. It can be described equivalently as 
\[
\mathcal{C}(\mathcal{I})=NP(J_1)\cap \cdots \cap NP(J_s).
\]
\end{thm}
\begin{proof}
We start by proving the last assertion.
Let $\mathcal{Q}=\bigcap_{i=1}^s NP(J_i)$.
First we see that for each $m\geq 1$ one has the containment $\frac{1}{m}NP(I_m)\subseteq \mathcal{Q}$. Indeed, since $ I_m=J_1^m\cap \cdots \cap J_s^m$, we have $NP(I_m)\subseteq  \bigcap_{i=1}^s NP(J_i^m)=m \cdot \mathcal{Q}$. This yields the inclusion $\mathcal{C}(\mathcal{I})\subseteq \mathcal{Q}$.

Next, for the opposite containment, we will show that for each $1\leq i\leq s$ every point of $\mathcal{Q} \cap \Q^n$ is in  $\mathcal{C}(\mathcal{I})$. This is enough to guarantee the containment $\mathcal{Q}\subseteq \mathcal{C}(\mathcal{I})$, since all the vertices of the former polyhedron have rational coordinates. Thus assume $\ba\in \mathcal{Q} \cap \Q^n$ and hence  $\ba\in NP(J_i) \cap \Q^n$ for all $1\leq i\leq s$. Fix $i$ and let $\bv_1,\ldots, \bv_t$ be the vertices of the polyhedron $NP(J_i)$. Since these correspond to a subset of the monomial generators of $J_i$ we notice that $\bv_j\in \Z^n$ for all $1\leq j\leq t$ and
\[
NP(J_i) = \conv\{\bv_1,\ldots, \bv_t\}+\R_{\geq 0}^n.
\]
By a version of Carath\'eodory's theorem for unbounded polyhedra \cite[Theorem 5.1]{cooper2017symbolic} we can write
\[
\ba=\sum_{j=1}^n \lambda_j\bv_{i_j} +\sum_{j=1}^n c_j \be_j,
\]
where $\lambda_j, c_j\geq 0$ are rational numbers satisfying $\sum_{j=1}^n \lambda_j=1$. Let $m$ be the least common multiple of the denominators of the rational numbers $ \lambda_j, c_j$ for all $1\leq j\leq n$. Multiplying the equation displayed above by $m$ we deduce the identity
\[
m\ba=\sum_{j=1}^n m\lambda_j\bv_{i_j} +\sum_{j=1}^n mc_j \be_j,
\]
where $\sum_{j=1}^n m \lambda_j=m$ and $m\lambda_j\in\N$ for all $1\leq j\leq n$. This yields that $\bx^{m\ba}\in J_i^m$ and since the argument holds for each $i$, we deduce that $\bx^{m\ba} \in \bigcap_{i=1}^s J_i^m=I_m$. Based on this we see that $\ba\in \frac{1}{m}NP(I_m)\subseteq \mathcal{C}(\mathcal{I})$, as desired.

The fact that $\mathcal{C}(\mathcal{I})$ is a polyhedron follows from the identity $\mathcal{C}(\mathcal{I})=\mathcal{Q}$ shown above because the latter is a finite intersection of polyhedra.
\end{proof}

As we show in the following two corollaries, the previous theorem allows us to identify the symbolic and irreducible polyhedra as asymptotic Newton polyhedra for the graded families of symbolic powers and irreducible powers of a monomial ideals respectively.

\begin{cor}
\label{cor:SPlimiting}
Let $I$ be a monomial ideal. Then the asymptotic Newton polyhedron of the family of symbolic powers $\{I^{(m)}\}_{m\geq 1}$ is the symbolic polyhedron $SP(I)$. 
\end{cor}
\begin{proof}
This follows by applying  \Cref{thm:limitshapeintcldecomp} to the family of symbolic powers, which is defined in terms of the decomposition $I=\bigcap_{P\in \Max(I)} Q_{\subseteq P}$ with $Q_{\subseteq P}=IR_P\cap R$. Together with \Cref{def:symbolicpolyhedron}, this result yields the claim. 
\end{proof}

\begin{cor}
\label{cor:IPlimiting}
Let $I$ be a monomial ideal. Then the asymptotic Newton polyhedron of the family $\{I^{\{m\}}\}_{m\geq 1}$ of irreducible powers is the irreducible polyhedron $IP(I)$. 
\end{cor}
\begin{proof}
By \Cref{def:irredpower}, we are in the setting of \Cref{thm:limitshapeintcldecomp} where the family of irreducible powers is defined in terms of a monomial irreducible decomposition $I=\bigcap_{i=1}^s J_i$. Thus \Cref{thm:limitshapeintcldecomp} and \Cref{def:irredpolyhedron} yield the desired conclusion.
\end{proof}

\section{Asymptotic invariants for families of monomial ideals}

\subsection{Asymptotic initial degrees and linear optimization}

In this section we define asymptotic invariants for graded families of monomial ideals which are derived from their initial degree. For a homogeneous ideal $I$ the initial degree, denoted $\alpha(I)$, is the least degree of a non zero element of $I$. Also termed the order of the ideal, this invariant references the position of $I$ in the topology given by the powers of the homogeneous maximal ideal $\m$, specifically $\alpha(I)$ is the largest integer such that $ I\subseteq \m^{\alpha(I)}$.
This interpretation enters into the picture in \Cref{homomaxhasleastalpha}.

\begin{defn}
\label{def:asymptoticinitialdegree}
For a graded family of ideals $\mathcal{I}=\{I_m\}_{m\geq 1}$ define the {\em asymptotic initial degree} of the family to be 
$\alpha(\mathcal{I})=\lim_{m\to \infty} \frac{\alpha(I_m)}{m}.$
\end{defn}

\begin{lem}
\label{rem:infimum}
For a graded family of ideals $\mathcal{I}=\{I_m\}_{m\geq 1}$ the limit $\lim_{m\to \infty} \frac{\alpha(I_m)}{m}$ exists and is equal to $\inf_{m\geq 1} \left\{\frac{\alpha(I_m)}{m}\right\}$.
\end{lem}
\begin{proof}
The existence of the limit is ensured by Fekete's lemma \cite{Fekete} by means of the subadditivity of the sequence of initial degrees $\{\alpha(I_m)\}_{m\geq 1}$. In turn, the subadditivity arises from the graded family property, as the containments $I_aI_b\subseteq I_{a+b}$ give rise to inequalities $\alpha(I_{a+b})\leq \alpha(I_a)+\alpha(I_b)$ for all integers $a,b\geq 1$. Fekete's lemma also gives that the limit in \Cref{def:asymptoticinitialdegree} is equal to the infimum of the respective sequence.
\end{proof}

Applying the definition for asymptotic initial degree of the family of symbolic powers recovers the notion of Waldschmidt constant introduced in \cite{Waldschmidt} and studied widely in the literature starting with the inspiring paper \cite{BoH}.

\begin{defn}
\label{def:Waldschmidt}
Let $I$ be a homogeneous ideal. The asymptotic initial degree of the family of symbolic powers $\{I^{(m)}\}_{m\geq 1}$ is termed the {\em Waldschmidt constant} of $I$ and defined as follows
\[
\widehat{\alpha}(I)=\lim_{m\to \infty} \frac{\alpha(I^{(m)})}{m}.
\]
\end{defn}

Applying the definition for asymptotic initial degree of the family of irreducible powers yields a novel invariant.
\begin{defn}
\label{def:naiveWaldschmidt}
Let $I$ be a monomial ideal. The asymptotic initial degree of the family of irreducible powers $\{I^{\{m\}}\}_{m\geq 1}$ is termed the {\em naive Waldschmidt constant} of $I$ and defined as follows
\[
\widetilde {\alpha}(I)=\lim_{m\to \infty} \frac{\alpha(I^{\{m\}})}{m}.
\]
\end{defn}

We now show that asymptotic initial degrees for families of monomial ideals are solutions to an optimization problem. The initial degree of a monomial ideal $I$ can be expressed as the solution of a linear programming problem in the following manner: 
\begin{equation}
\label{eq:alphaLP}
\alpha(I)=\min\{y_1+\cdots+y_n \mid (y_1,\cdots, y_n)\in NP(I)\}.
\end{equation}
This is because the optimal solution is attained at a vertex of $NP(I)$ and the vertices of $NP(I)$ correspond to a subset of the minimal generators of $I$.
We see below that the asymptotic initial degree for a graded family of monomial ideals can also be expressed as an optimization problem. Moreover, the feasible set is the limiting body of the family as defined in \Cref{def:limitingshape}. 

\begin{thm}
Let $\mathcal{I}=\{I_m\}_{m\geq 1}$ be a graded family of monomial ideals. Then $\alpha(\mathcal{I})$ is the solution of the following optimization problem
\begin{center}
\begin{tabular}{rl}
\label{LP}
\emph{minimize} & $y_1+\cdots+y_n $\\
\emph{subject to} & $(y_1,\cdots, y_n)\in \overline{\mathcal{C}(\mathcal{I})}$,
\end{tabular}
\end{center}
where $\overline{\mathcal{C}(\mathcal{I})}$ denotes the closure of $\mathcal{C}(\mathcal{I})$ in the Euclidean topology of $\R^n$.
\end{thm}
\begin{proof}
Recall from \Cref{rem:infimum} the alternate definition $\alpha(\mathcal{I})=\inf_{m\geq 1} \frac{\alpha(I_m)}{m}$.
From \eqref{eq:alphaLP} we deduce 
$
\alpha(I_m)= \min\{y_1+\cdots+y_n \mid (y_1,\cdots, y_n)\in NP(I_m)\}
$, hence there are equalities
\[
\frac{\alpha(I_m)}{m}= \min\{y_1+\cdots+y_n \mid (y_1,\cdots, y_n)\in \frac{1}{m}NP(I_m)\}.
\]
Now passing to the infimum and  denoting the solution of the optimization problem in the statement of the theorem by $\beta$, we deduce 
\begin{eqnarray*}
\alpha(\mathcal{I}) =\inf_{m\geq 1} \frac{\alpha(I_m)}{m}&=& \inf_{m\geq 1}\left \{\min\{y_1+\cdots+y_n \mid (y_1,\cdots, y_n)\in \frac{1}{m}NP(I_m)\}\right\}\\
&=& \inf\{y_1+\cdots+y_n \mid (y_1,\cdots, y_n)\in \bigcup_{m\geq 1} \frac{1}{m}NP(I_m)=\mathcal{C}(\mathcal{I})\}\\
&=& \min\{y_1+\cdots+y_n \mid (y_1,\cdots, y_n)\in\ov{\mathcal{C}(\mathcal{I})}\}=\beta.
\end{eqnarray*}
\end{proof}

Applying this theorem, we are able to recover a result relating the Waldschmidt constant to the symbolic polyhedron from \cite[Corollary 6.3]{cooper2017symbolic} and \cite[Theorem 3.2]{bocci2016waldschmidt}

\begin{cor}
\label{cor:WaldschmidtLP}
The Waldschmidt constant of a monomial ideal $I$ is the solution to the following linear optimization problem with feasible region given by its symbolic polyhedron: 
\begin{center}
\begin{tabular}{rl}
\label{LPWaldschmidt}
\emph{minimize} & $y_1+\cdots+y_n $\\
\emph{subject to} & $(y_1,\cdots, y_n)\in SP(I)$.
\end{tabular}
\end{center}
\end{cor}

\begin{cor}
\label{cor:naiveWaldschmidtLP}
The naive Waldschmidt constant of a monomial ideal $I$ is the solution to the following linear optimization problem with feasible region given by its irreducible polyhedron:
\begin{center}
\begin{tabular}{rl}
\label{LPnaiveWaldschmidt}
\emph{minimize} & $y_1+\cdots+y_n $\\
\emph{subject to} & $(y_1,\cdots, y_n)\in IP(I)$.
\end{tabular}
\end{center}
\end{cor}

From the containments in \Cref{thm:naivecontainment} and the above two corollaries we deduce inequalities relating the various asymptotic initial degrees.

\begin{prop}
\label{prop:ineq}
Every monomial ideal $I$ satisfies the inequality $\widetilde{\alpha}(I)\leq \widehat{\alpha}(I)\leq \alpha(I)$.
\end{prop}
\begin{proof}
 \Cref{thm:naivecontainment} gives $NP(I)\subseteq SP(I) \subseteq IP(I)$ and taking the minimum value of the sum of the coordinates of any point in these convex bodies turns containments into reverse inequalities. These minimum values are $\widetilde{\alpha}(I)$ for $SP(I)$ and $\widehat{\alpha}(I)$ for $IP(I)$ by \Cref{cor:WaldschmidtLP} and
\Cref{cor:naiveWaldschmidtLP} respectively and  $\alpha(I)$ for $NP(I)$ by equation \eqref{eq:alphaLP}.
\end{proof}

%


%

\begin{ex}
\label{ex:4.9}
The inequalities $\widetilde{\alpha}(I)\leq \widehat{\alpha}(I)\leq \alpha(I)$ are in general strict. For the ideal $I=(x^2,xy,y^2)$ in \Cref{ex:1} one finds by applying the above corollaries $\widetilde{\alpha}(I)= \frac{4}{3} < 2=\widehat{\alpha}(I)  =\alpha(I)$. The value of  $\widetilde{\alpha}(I)$ follows by observing that, as illustrated in \Cref{ex:1}, $IP(I)$ has vertices at $(2,0),(0,2)$ and $(\frac{2}{3},\frac{2}{3})$. The latter furnishes the solution to the linear program in \Cref{cor:naiveWaldschmidtLP}. 

For the ideal $I=(xy,xz,yz)$ in \Cref{ex:2} one finds by applying the above corollaries $\widetilde{\alpha}(I)= \widehat{\alpha}(I) =\frac{3}{2}<2 =\alpha(I)$. The value of  $\widehat{\alpha}(I)$ follows by observing that, as illustrated in \Cref{ex:2}, $SP(I)$ has vertices at $(1,1,0), (1,0,1), (0,1,1)$ and $(\frac{1}{2},\frac{1}{2},\frac{1}{2})$. The latter furnishes the solution to the linear program in \Cref{cor:WaldschmidtLP}. 

\end{ex}

Under special circumstances, we may also deduce equality between the asymptotic invariants discussed above.

\begin{prop}
If \(I\) is a monomial ideal whose irredundant irreducible components have distinct radicals, then it satisfies \(\widehat{\alpha}(I)=\widetilde{\alpha}(I)\). In particular, this equality holds when $I$ is square-free.
\end{prop}
\begin{proof}
The equality follows from \Cref{cor:WaldschmidtLP} and
\Cref{cor:naiveWaldschmidtLP}   after noticing that $SP(I)=IP(I)$ under the given hypothesis, according to \Cref{ex:sqfree}.\end{proof}

\subsection{Lower bounds on asymptotic initial degrees}

\Cref{prop:ineq} establishes that the initial degree of $I$ is an upper bound for both $\widetilde{\alpha}(I)$ and $\widehat{\alpha}(I)$. This upper bound is attained, for example, when $I$ is an irreducible monomial ideal, hence a complete intersection, and thus $I^{\{m\}}=I^{(m)}=I^m$ for each integer $m\geq 1$. 

We now discuss lower bounds for the asymptotic invariants $\widetilde{\alpha}(I)$ and $\widehat{\alpha}(I)$. These are formulated in terms of the initial degree of $I$ and an invariant termed {\em big-height}, which is defined as follows:
\[
\bight(I)=\max\{\het(P)\mid P\in \Ass(R/I)\}.
\]
This invariant can be computed from an irredundant primary decomposition and in particular also from an irredundant irreducible decomposition of $I$ as the maximum height of the primary, respectively irreducible, ideals appearing in the decomposition.

For the Waldschmidt constant the following lower bounds are either known or conjectured to be true. An inequality similar to \Cref{prop:Skoda} first appeared in \cite{Skoda, Waldschmidt} and was proven in the generality given here in \cite{HaHu}.

\begin{prop}[Skoda bound]
\label{prop:Skoda}
Every homogeneous ideal $I$ satisfies the following inequality 
\[
\widehat{\alpha}(I)\geq \frac{\alpha(I)}{\bight(I)}.
\]
\end{prop}

The following conjecture proposing a stronger bound has been formulated in \cite[Conjecture 6.6]{cooper2017symbolic}.
\begin{conj}[Chudnovsky bound]
\label{conj:Chudmon}
Every monomial ideal $I$ satisfies the inequality 
\[
\widehat{\alpha}(I)\geq \frac{\alpha(I)+\bight(I)-1}{\bight(I)}.
\]
\end{conj}

The Chudnovsky bound in \Cref{conj:Chudmon} is known to hold true for square-free monomial ideals cf. \cite[Theorem 5.3]{bocci2016waldschmidt}. 

We now proceed to establish lower bounds for the asymptotic irreducible degree $\widetilde{\alpha}(I)$, by analogy to the bounds discussed above for $\widehat{\alpha}(I)$. First we prove a Skoda-type lower bound.

\begin{thm}
\label{SkodaTypeBound}
Every monomial ideal $I$ satisfies the inequality $\widetilde{\alpha}(I)\geq\frac{\alpha(I)}{\bight(I)}$.
\end{thm}

\begin{proof} We proceed by adapting the proof of \cite[Theorem 5.3]{bocci2016waldschmidt}. 

Let $I$ be a monomial ideal with  $\bight(I)=e$ and irredundant irreducible decomposition $I = J_1 \cap \cdots \cap J_s$. We know from \Cref{cor:naiveWaldschmidtLP} that $\widetilde{\alpha}(I)$ is the minimum value of $y_1 + \cdots + y_n$ over $IP(I)$ and from \Cref{rem:eqNP(Q)} that, if for each $i = 1, \dotsc, s$ $J_i=(x_{1}^{a_{i1}},\ldots, x_{n}^{a_{i n}})$, then the bounding inequalities for this polyhedron are
$$
IP(I)=\begin{cases}
\frac{1}{a_{11}}y_{1} + \cdots + \frac{1}{a_{1n}} y_{n} \geq 1\\
\cdots \\
\frac{1}{a_{s1}}y_{1} + \cdots + \frac{1}{a_{sn}} y_{n} \geq 1\\
y_1,\ldots, y_n\geq 0
\end{cases}.
$$

To establish the claim, if suffices to show that, for every ${\bf t} \in IP(I)$, we have $$t_1 + \cdots + t_n \geq \frac{\alpha(I)}{\bight(I)}= \frac{\alpha(I)}{e}$$ which implies by taking infimums that $\widetilde{\alpha}(I)$, the minimal value of the sum of coordinates of any point in $IP(I)$, will  satisfy the desired inequality. 

We find a subset of the components of ${\bf t}$ whose sum is greater or equal to $\alpha(I)/e$. 

To start, consider a bounding inequality corresponding to an irreducible component $J_{i_1}$. This takes the form 
\[ \frac{1}{a_{i_11}}y_{1} + \cdots + \frac{1}{a_{i_1n}} y_{n} \geq 1,\]
where  the number of $y_i$ whose coefficients are non zero in the preceding inequality is the height of the monomial prime ideal $\sqrt{J_i}$, thus at most $e$. 
The displayed inequality thus implies that for $\by=\bt$ at least one of the terms is greater or equal to $\frac{1}{e}$, i.e., 
 \[t_{k_1} \geq \frac{a_{i_1k_1}}{e} \text{ for some } 1\leq k\leq n  \text{ and some integer } a_{i_1k_1}\geq 1. 
 \]
 
Now, suppose we have found $t_{k_1}, t_{k_2}, \dots, t_{k_m}$ such that $t_{k_1} \geq \frac{a_{i_1k_1}}{e}$, ..., $t_{k_m} \geq \frac{a_{i_mk_m}}{e}$, but we  have $a_{i_1k_1}+ a_{i_2k_2} + \cdots + a_{i_m k_m} < \alpha(I)$. Consider the monomial $x_{k_1}^{a_{i_1k_1}} x_{k_2}^{a_{i_1k_2}} \cdots x_{k_m}^{a_{i_mk_m}}$. By the assumption, it has degree smaller than $\alpha(I)$, so it's not an element of $I$. Therefore, there is some component $J_{i_{m+1}}$ that does not contain this monomial. Repeating the previous argument, from the corresponding inequality we obtain  $t_{k_{m+1}} \geq \frac{a_{i_{m+1}k_{m+1}}}{e}$ for some $a_{i_{m+1}k_{m+1}}\geq 1$. There are two possibilities depending on whether $k_{m+1}\in \{k_1,\ldots, k_m\}$ or not:

\begin{enumerate}
    \item $t_{k_{m+1}}$ is not one of $t_{k_1}, t_{k_2}, \dots, t_{k_m}$. Then we observe that 
    \[a_{i_1k_1}+ a_{i_2k_2} + \cdots + a_{i_mk_m} + a_{i_{m+1}k_{m+1}} > a_{i_1k_1}+ a_{i_2k_2} + \cdots + a_{i_mk_m} .\]
    
    \item $t_{k_{m+1}}$ is one of $t_{k_1}, t_{k_2}, \dots, t_{k_m}$, say $t_{k_{m+1}} = t_{k_\ell}$. Since the monomial $x_{k_1}^{a_{i_1k_1}} x_{k_2}^{a_{i_2k_2}} \cdots x_{k_m}^{a_{i_mk_m}}$  is not contained in $J_{i_{m+1}}$, it must be that $a_{i_{m+1}k_{m+1}}> a_{i_\ell k_\ell}$. Therefore, we can replace the inequality $t_{k_\ell} \geq \frac{a_{i_\ell k_\ell}}{e}$ by the stronger inequality $t_{k_{m+1}}\geq \frac{a_{i_{m+1}k_{m+1}}}{e}$. Specifically, redefining $i_\ell:=i_{m+1}$ and $k_\ell:=k_{m+1}$ and thus $a_{i_\ell k_\ell}:=a_{i_{m+1} k_{m+1}}$ increases the value of the sum
     $a_{i_1k_1}+ a_{i_2k_2} + \cdots + a_{k_m}$.
      \end{enumerate}
 Since in either case the value of the sum  $a_{i_1k_1}+ a_{i_2k_2} + \cdots + a_{i_mk_m} + a_{i_{m+1}k_{m+1}}$ or $a_{i_1k_1}+ a_{i_2k_2} + \cdots + a_{i_mk_m}$ increases, we see that iterating this procedure eventually results in positive integers $a_{i_1k_1}, a_{i_2k_2}, \cdots + a_{i_mk_m} $ such that 
 \[a_{i_1k_1}+ a_{i_2k_2} + \cdots + a_{i_mk_m} \geq \alpha(I)\] as well as in a corresponding set of coordinates of ${\bf t}$  that  satisfy the desired inequality
$$t_{k_1} + \cdots + t_{k_m} \geq \frac{a_{i_1k_1}+ a_{i_2k_2} + \cdots + a_{i_mk_m} }{e} \geq \frac{\alpha(I)}{e}.$$
\end{proof}

We remark that the direct analogue of the Chudnovsky bound in \Cref{conj:Chudmon} fails for $\widetilde{\alpha}(I)$, as shown by the following example.

\begin{ex}
Consider the ideal $I=(x^2,xy,y^2)=(x^2,y)\cap(x,y^2)\subseteq k[x,y]$. The  initial degree is $\alpha(I)=2$, the big height is $\bight(I)=2$ and the naive Waldschmidt constant is $\widetilde{\alpha}(I)=\frac{4}{3}$ per \Cref{ex:4.9}. This gives an inequality
\[
\widetilde{\alpha}(I)=\frac{4}{3}<\frac{3}{2}=\frac{\alpha(I)+\bight(I)-1}{\bight(I)}.
\]
\end{ex}

However, there are many  ideals for which the expression in the Chudnovsky conjecture \Cref{conj:Chudmon} does indeed provide a lower bound on $\widetilde{\alpha}(I)$. In the next section we give a modified Chudnovsky-type lower bound for $\widetilde{\alpha}(I)$ that applies to all monomial ideals $I$.

\subsection{Powers of the maximal ideal}

In this section we determine the naive Waldschmidt constant for the powers of the homogeneous maximal ideal.
We will later use this to deduce a Chudnovsky-type lower bound on the naive Waldschmidt constant of ideals primary to the homogeneous maximal ideal.

In the following, we denote by $\mathfrak{m}_n$ the homogeneous maximal ideal $(x_1,\ldots, x_n)$ of the polynomial ring  $R=k[x_1,\dots,x_n]$. We start by establishing the irredundant irreducible decompositions for the ordinary powers of $\m_n$.
\begin{notation}
For each positive integer $s$ we denote by $P_n(s)$ be the set of partitions of $s$ into $n$ nonempty parts
$$
P_n(s)=\left\{(a_1,\dots,a_n)\left|\,a_i\in\N,a_i\geq1,\sum_{i=1}^n a_i=s\right.\right\}.
$$
\end{notation}


\begin{prop}
\label{homomaxidealQdecomp}
Given an integer $d\geq 1$, the irredundant irreducible decomposition of the ideal $\mathfrak{m}_n^d=(x_1,\ldots,x_n)^d$ is
\begin{equation}
\label{eq:irreddecomp}
\mathfrak{m}_n^d=\bigcap_{(a_1,\dots,a_n)\in P_n(d+n-1)}(x_1^{a_1},\dots,x_n^{a_n})
\end{equation}
\end{prop}
\begin{proof}
Let $\textbf{x}^\textbf{b}=x_1^{b_1}\cdots x_n^{b_n}\in\bigcap_{(a_1,\dots,a_n)\in P_n(d+n-1)}(x_1^{a_1},\dots,x_n^{a_n})$, and suppose $\textbf{x}^\textbf{b}\not\in\mathfrak{m}_n^d$. Then there are inequalities
$$
\sum_{i=1}^n b_i<d \qquad\text{ and thus }\qquad d-\sum_{i=1}^n b_i\geq 1.
$$
Let
$$
a_i=
\begin{cases}
    b_i+1&1\leq i<n\\
    b_n+d-\sum_{i=1}^{n} b_i&i=n
\end{cases}
$$
which implies $a_i\geq b_i+1$ for all $1\leq i\leq n$. From this, we have an equality
\[
\sum_{i=1}^n a_i=\sum_{i=1}^{n-1}(b_i+1)+b_n+d-\sum_{i=1}^{n} b_i=\sum_{i=1}^{n}b_i+n-1+d-\sum_{i=1}^{n}b_i=d+n-1
\]
and hence $(a_1,\dots,a_n)\in P_n(d+n-1)$. But since $a_i>b_i$ for all $i$, $\textbf{x}^\textbf{b}\not\in(x_1^{a_1},\dots,x_n^{a_n})$, a contradiction. As a result, we obtain the containment
\begin{equation}
\label{eq:1.1}
    \bigcap_{(a_1,\dots,a_n)\in P_n(d+n-1)}(x_1^{a_1},\dots,x_n^{a_n})
    \subseteq \mathfrak{m}_n^d
\end{equation}
Now take $\textbf{x}^\textbf{b}\in\mathfrak{m}_n^d$, and suppose $\textbf{x}^\textbf{b}\not\in \bigcap_{(a_1,\dots,a_n)\in P_n(d+n-1)}(x_1^{a_1},\dots,x_n^{a_n})$. Then there is some $Q=(x_1^{c_1},\dots,x_n^{c_n})$ with $(c_1,\dots,c_n)\in P_n(d+n-1)$ such that $\textbf{x}^\textbf{b}\not\in Q$. This implies that $c_i>b_i$ for all $1\leq i\leq n$, so $c_i\geq b_i+1$. But then we deduce
\[
d+n-1=\sum_{i=1}^n c_i
\geq \sum_{i=1}^n (b_i+1)
=n+\sum_{i=1}^n b_i \geq n+d
>d+n-1,
\]
which is of course a contradiction. Hence
\begin{equation}
\label{eq:1.2}
\mathfrak{m}_n^d\subseteq\bigcap_{(a_1,\dots,a_n)\in P_n(d+n-1)}(x_1^{a_1},\dots,x_n^{a_n})
\end{equation}
Combining \eqref{eq:1.1} and \eqref{eq:1.2}, we obtain our desired result.
\end{proof}

Having determined the irredundant irreducible decomposition of $\m_n^d$, we deduce the bounding inequalities for the irreducible polyhedron from \Cref{rem:eqNP(Q)}.

\begin{cor}
\label{cor:boundingIPmnd}
The irreducible polyhedron of the ideal $\m_n^d$ is given by the inequalities
\[
\begin{cases}
\frac{1}{a_1}y_1+\cdots+ \frac{1}{a_n}y_n\geq 1 &\text{ for } (a_1,\dots,a_n)\in P_n(d+n-1)\\
y_i\geq 0 &\text{ for } 1\leq i\leq n.
\end{cases}
\]
\end{cor}

Next we give closed formulas for the naive Waldschmidt constant for the powers of the maximal ideal. We first single out the case when this value is an integer.

\begin{prop}
\label{mod1formula}
Suppose $d\equiv 1\mod n$. Then the naive Waldschmidt constant of $\m_n^d$ is 
$$
\widetilde{\alpha}(\mathfrak{m}_n^d)=\frac{d+n-1}{n}\in\mathbb{N}.
$$
\end{prop}
\begin{proof}
If $d\equiv 1\mod n$, then $d+n-1$ is an integer multiple of $n$; in other words, $\frac{d+n-1}{n}$ is an integer, say $m$. The ideal $(x_1^m,\dots,x_n^m)$ is in the irreducible decomposition of $\mathfrak{m}_n^d$ by \Cref{homomaxidealQdecomp}. The bounding inequality corresponding to $NP(x_1^m,\dots,x_n^m)$ 
$$
\frac{1}{m}y_1+\cdots+\frac{1}{m}y_n\geq 1
\quad\Rightarrow\quad
y_1+\cdots+y_n\geq m
$$
indicates that $\widetilde{\alpha}(\mathfrak{m}_n^d)\geq m$. Consider the vector $(\frac{m}{n},\dots,\frac{m}{n})$ in $\mathbb{R}^n$ that clearly has  sum of coordinates $m$. For each component $(x_1^{a_1},\dots,x_n^{a_n})$ in the irreducible decomposition there is an identity
$$
\begin{aligned}
\frac{1}{a_1}\left(\frac{m}{n}\right)+\cdots+\frac{1}{a_n}\left(\frac{m}{n}\right)
&=\frac{m}{n}\sum_{i=1}^n \frac{1}{a_i}.
\end{aligned}
$$
The value $\frac{1}{n}\sum_{i=1}^n \frac{1}{a_i}$ is the inverse of the harmonic mean of the set $a_1,\dots,a_n$ and the arithmetic mean for this set is $m$. Hence the inequality relating these means yields
$$
\frac{1}{a_1}\left(\frac{m}{n}\right)+\cdots+\frac{1}{a_n}\left(\frac{m}{n}\right)
\geq m\left(\frac{1}{m}\right)= 1.
$$
Therefore the point $(\frac{m}{n},\dots,\frac{m}{n})$ is part of the Newton polyhedron of each irreducible component of $\mathfrak{m}_n^d$, i.e., $(\frac{m}{n},\dots,\frac{m}{n})\in IP(\mathfrak{m}_n^d)$. Since it was shown before that the least value of the sum of coordinates of points in this polyhedron is at least $m$, and the point identified above has  sum of coordinates exactly $m$, we conclude that $\widetilde{\alpha}(\mathfrak{m}_n^d)= m$.
\end{proof}

\begin{rem}
\label{rem:mod1formula}
The right hand side in the equality displayed in \Cref{mod1formula} matches the Chudnovsky lower bound $\frac{\alpha(\m_n)+\bight(\m_n)-1}{\bight(\m_n)}$; see \Cref{conj:Chudmon}.
\end{rem}

Before we continue our analysis, we state a simple fact that will become useful later. The proof is omitted, since it is a direct verification.
\begin{lem}
\label{recipsum}
If $x,y\in\mathbb{R}_{>0}$ are such that $x\geq y+1$, then 
$
\frac{1}{x}+\frac{1}{y}\geq\frac{1}{x-1}+\frac{1}{y+1}.
$
\end{lem}
An interesting consequence of the above lemma is presented below.
\begin{prop}
\label{cor:recipineq}
Fix an integer $s>0$. The minimum value of the function $f(\ba)=\frac{1}{a_1}+\cdots + \frac{1}{a_n}$, where the tuple $\ba=(a_1,\ldots, a_n)$ ranges over $P_n(s)$ is attained by a partition where the parts differ by at most one, that is, $|a_i-a_j|\leq 1$ for all $1\leq i<j\leq n$.
\end{prop}
\begin{proof}
The result follows by noticing that modifying a partition in a manner that decreases the difference between the parts results in an increase of the objective function $f$. Indeed,  \Cref{recipsum} insures that if $(a_1,\ldots, a_n)\in P_n(s)$ has two parts $a_i, a_j$ such that $|a_i-a_j|>1$, then the partition $(a'_1,\ldots, a'_n)\in P_n(s)$ obtained by setting $a'_k=a_k$ whenever $k\not\in \{i,j\}$, $a'_i=\max\{a_i,a_j\}-1$, $a'_j=\min\{a_i,a_j\}+1$ satisfies
\[ f(\ba)=\sum_{\ell=1}^n \frac{1}{a_\ell} \geq \sum_{\ell=1}^n \frac{1}{a'_\ell} =f(\ba'). \]
\end{proof}


Now we turn to the determination of $\widetilde{\alpha}(\mathfrak{m}_n^d)$ for arbitrary values of $d$.

\begin{thm}
\label{modkformula}
Suppose $d$ is a positive integer and $d-1\equiv k\mod n$, $0\leq k<n$. Then 
the following equality holds
\[
\widetilde{\alpha}(\mathfrak{m}_n^d)
=\frac{(n+d-1-k)(2n+d-1-k)}{n(2n+d-1-2k)}.
\]
\end{thm}
\begin{proof}
First, if $k=0$, then the formula becomes
$$
\widetilde{\alpha}(\mathfrak{m}_n^d)=
\frac{(2n+d-1-0)(n+d-1-0)}{n(2n+d-1-2(0))}
=\frac{n+d-1}{n},
$$
which is in accordance with \Cref{mod1formula}. Therefore, let us consider the case $k>0$.

Let $a=\left\lceil \frac{n+d-1}{n}\right\rceil$ and $b=\left\lfloor \frac{n+d-1}{n}\right\rfloor$, with explicit expressions$$
a=\frac{n+d-1+(n-k)}{n},\qquad b=\frac{n+d-1-k}{n}.
$$
Thus $a$ and $b$ are positive integers. We define the \textit{balanced} partition of $n+d-1$ as the unordered $n$-tuple where $k$ of the elements are $a$ and $n-k$ of the elements are $b$. This partition is in $P_n(n+d-1)$ since these elements sum to $n+d-1$. 


Consider now the components of the irreducible decomposition \eqref{eq:irreddecomp} corresponding to permutations of this balanced partition. There are ${n\choose k}$ such irreducible components; each corresponds to a permutation $\sigma$ in the symmetric group on $n$ elements in the following way $$
J_\sigma=(x_{\sigma(1)}^a,\ldots, x_{\sigma(k)}^a,x_{\sigma(k+1)}^b,\ldots,x_{\sigma(n)}^b).
$$
There are multiple permutations $\sigma$ which give the same  irreducible component $J_\sigma$.
The bounding inequality for $IP(\m_n^d)$ corresponding to the component $J_\sigma$ is
$$
\frac{1}{a}\left(y_{\sigma(1)}+\cdots+y_{\sigma(k)}\right)+\frac{1}{b}\left(y_{\sigma(k+1)}+\cdots+y_{\sigma(n)}\right)\geq 1
$$
Summing up these inequalities for all the distinct ideals $J_\sigma$ and utilizing the symmetry of the coefficients yields
%
\[
(y_1+y_2+\cdots+y_n)\left({n-1\choose k-1}\frac{1}{a}+{n-1\choose k}\frac{1}{b}\right) \geq{n\choose k}
\]
whence we deduce that  any point $\by=(y_1,\ldots, y_n)\in IP(I)$ satisfies
\[
y_1+y_2+\cdots+y_n\geq \frac{{n\choose k}}{{n-1\choose k-1}\frac{1}{a}+{n-1\choose k}\frac{1}{b}} =\frac{(n+d-1-k)(2n+d-1-k)}{n(2n+d-1-2k)}
:=\beta.
\]
From \Cref{cor:naiveWaldschmidtLP} we now deduce the inequality
$
\widetilde{\alpha}(\mathfrak{m}_n^d)\geq\beta.
$

Next consider the vector $\widetilde{\by}\in \R^n$ having each component  $\widetilde{y}_i=\beta/n$. We show that $\widetilde{\by}\in IP(\m_n^d)$ by verifying that this vector satisfies the bounding inequalities in \Cref{cor:boundingIPmnd}.  Given $(a_1,\ldots, a_n)\in P_n(d+n-1)$ there is an equality
\[
\frac{1}{a_1}\widetilde{y}_1+\cdots+\frac{1}{a_n}\widetilde{y}_n=\left(\frac{1}{a_1}+\cdots+\frac{1}{a_n}\right)\cdot \frac{\beta}{n}\]
and by \Cref{cor:recipineq} we can compare the sum of the reciprocals for the partition $(a_1,\ldots, a_n)$ to that of the balanced partition as follows
\[ \frac{1}{a_1}+\cdots+\frac{1}{a_n}\geq k\cdot \frac{1}{a}+(n-k)\cdot \frac{1}{b}=\frac{n^2(2n+d-1-2k)}{(n+d-1-k)(2n+d-1-k)}= \frac{n}{\beta}.\]
Altogether, the previous two displayed equations yield the inequality
\[
\frac{1}{a_1}\widetilde{y}_1+\cdots+\frac{1}{a_n}\widetilde{y}_n\geq \frac{n}{\beta}\cdot \frac{\beta}{n}=1.
\]
%
%
Since we have shown $\widetilde{\by}$ satisfies the bounding inequalities for the irreducible polyhedron of $\mathfrak{m}_n^d$, it follows that  $\widetilde{\by}\in IP(\m_n^d)$ and thus 
\[
\widetilde{\alpha}(\mathfrak{m}_n^d)\leq \widetilde{y}_1+\cdots +\widetilde{y}_n=\beta=\frac{(n+d-1-k)(2n+d-1-k)}{n(2n+d-1-2k)}. 
\]
Together with the opposite inequality proven above this finishes the proof.

%

\end{proof}

We give a lower bound that extends \Cref{rem:mod1formula}.

\begin{cor}
\label{cor:tight}
Let $d,n$ be positive integers. Then the following inequality holds
\[
\widetilde{\alpha}(\m_n^d)\geq \left\lfloor \frac{d+n-1}{n}\right\rfloor,
\]
with equality taking place if and only if $d\equiv 1 \pmod{n}$.
\end{cor}
\begin{proof}
In view of \Cref{modkformula}, setting $d-1\equiv k\pmod{n}$ where $0\leq k\leq n-1$, the claim is equivalent to the following easily verified inequality
\[
\frac{(n+d-1-k)(2n+d-1-k)}{n(2n+d-1-2k)}\geq \frac{d+n-1-k}{n}=\left\lfloor \frac{d+n-1}{n}\right\rfloor.
\]
\end{proof}

In view of the result above, we make a conjecture regarding the naive Waldschmidt constant that parallels \Cref{conj:Chudmon}.

\begin{conj}
\label{conj:weakChudnovski}
Let $I$ be a monomial ideal. Then the following  inequality holds
\[
\widetilde{\alpha}(I)\geq \left\lfloor \frac{\alpha(I)+\bight(I)-1}{\bight(I)}\right\rfloor.
\]
\end{conj}

In \Cref{LowerBoundThm} we prove this conjecture for the case when $I$ has maximum possible big-height, namely $\bight(I)=n$. The importance of determining the value of the naive Waldschmidt constant for the powers of the homogeneous maximal ideal earlier in this section becomes apparent in the next result because this provides lower bounds for the naive Waldschmidt constant of arbitrary ideals.

\begin{thm}
\label{homomaxhasleastalpha}
Let $I$ be a monomial ideal in $K[x_1,\ldots,x_n]$ with $\alpha(I)=d$. Then the  inequality $\widetilde{\alpha}(\mathfrak{m}_n^d)\leq\widetilde{\alpha}(I)$  holds.
\end{thm}

\begin{rem}
The analogue of the above theorem fails when replacing the naive Waldschmidt constant with the Waldschmidt constant. That is, if $\alpha(I)=d$,  the  inequality $\widehat{\alpha}(\mathfrak{m}_n^d)\leq\widehat{\alpha}(I)$ need not hold. This can be seen taking $I=(xy, xz, yz)$, an ideal which satisfies the containment $I\subseteq \m_3^2$, but yields $\widehat{\alpha}(I)=\frac{3}{2}<\widehat{\alpha}(\m_3^2)=2$. 

It is nevertheless true that for square-free monomial ideals $I\subseteq J$ one has $\widehat{\alpha}(I)\geq \widehat{\alpha}(J)$; see \cite[Lemma 3.10]{schweig}. Our proof for \Cref{homomaxhasleastalpha} draws inspiration from this result. Before giving the proof, we require some additional preparation.
\end{rem}

\begin{defn}
\label{def:irr}
For an ideal $I$ denote
$$
\Irr(I):=\{J\,|\,J\text{ is irreducible and } I\subseteq J\}.
$$
\end{defn}

For any monomial ideal $I$, the set $\Irr(I)$ is a partially ordered set with respect to containment which has finitely many minimal elements. Moreover, $J_1, \ldots, J_s$ are the minimal elements of $\Irr(I)$ with respect to containment if and only if $I=J_1\cap \cdots \cap J_s$ is the irredundant irreducible decomposition of $I$.

\begin{lem}
\label{lem:homomaxhasleastalpha}
If $I\subseteq I'$ are monomial ideals, then the following hold:
\begin{enumerate}
\item $\Irr(I')\subseteq \Irr(I)$,
\item if $J'$ is a minimal element of $\Irr(I')$ with respect to containment then there exists a minimal element $J\in \Irr(I)$ with respect to containment such that $J\subseteq J'$, 
\item $\widetilde{\alpha}(I)\geq \widetilde{\alpha}(I')$.
\end{enumerate}
\end{lem}
\begin{proof}
The containment $\Irr(I')\subseteq \Irr(I)$ follows from \Cref{def:irr} and the fact that $I\subseteq I'$.

Suppose $J'$ is minimal in $\Irr(I')$. Consider the set $S=\{J\in \Irr(I)\,|\,J\subseteq J'\}$. This is a non-empty subset of $\Irr(I)$ since $J'\in S$. Thus it has a minimal element with respect to containment, let's call it $J$. Moreover, since $S$ is a lower interval of the poset $\Irr(I)$, we deduce that $J$ is in fact a minimal element of $\Irr(I)$. 

Now let $I=J_1\cap \cdots \cap J_s$ and $I'=J'_1\cap \cdots \cap J'_t$ be the irredundant irreducible decompositions for $I$ and $I'$ respectively. From the second assertion of this lemma, for every  $j\in\{1,2,\ldots, t\}$ there exists an $i_j\in\{1,\ldots, s\}$ such that $J_{i_j}\subseteq J'_j$. From this we deduce $NP(J_{i_j})\subseteq NP(J'_j)$ for each $j$ and these containments combine to show the following 
$$
IP(I)=\bigcap_{i=1}^s NP(J_i) \subseteq \bigcap_{j=1}^t NP(J_{i_j})\subseteq  \bigcap_{j=1}^t NP(J'_j)= IP(I').
$$
Having established the containment $IP(I)\subseteq IP(I')$ above, we deduce from this containment and \Cref{cor:naiveWaldschmidtLP} the desired inequality $\widetilde{\alpha}(I)\geq \widetilde{\alpha}(I')$.
\end{proof}

\begin{proof}[Proof of \Cref{homomaxhasleastalpha}]
\Cref{homomaxhasleastalpha} follows from part 3 of \Cref{lem:homomaxhasleastalpha} applied to $I'=\m_n^d$. The containment $I\subseteq I'=\m_n^d$ is ensured by the hypothesis $\alpha(I)=d$.
\end{proof}

The following consequence of \Cref{homomaxhasleastalpha} establishes a lower bound on the naive Waldschmidt constant applicable to  all monomial ideals.

\begin{thm}
\label{LowerBoundThm}
Let $I\subseteq K[x_1,\dots,x_n]$ be a monomial ideal with $\alpha(I)=d$. If $d-1\equiv k\mod(n)$, $0\leq k<n$, then the following inequalities hold
$$ 
\alpha(I)\geq \widehat{\alpha}(I)
\geq \widetilde{\alpha}(I)
\geq\frac{(n+d-1-k)(2n+d-1-k)}{n(2n+d-1-2k)} \geq \left \lfloor \frac{\alpha(I)+n-1}{n}\right \rfloor.
$$
\end{thm}
\begin{proof}
This follows from \Cref{prop:ineq}, \Cref{homomaxhasleastalpha} and \Cref{modkformula}.
\end{proof}

The above inequalities establish  the validity of \Cref{conj:weakChudnovski} for monomial ideals $I$ which have the maximal ideal as an associated prime. For this class of ideals,  \Cref{conj:Chudmon} is obviously satisfied as well, since the symbolic and ordinary powers agree and thus $\widehat{\alpha}(I)=\alpha(I)$.

%

\subsection*{Acknowledgements}
The second author was supported by the NSF RTG grant in algebra and combinatorics at the University of Minnesota  DMS--1745638. The sixth author was supported by NSF DMS--2101225. This work was completed in the framework of the 2020 Polymath Jr.~program \href{https://geometrynyc.wixsite.com/polymathreu}{https://geometrynyc.wixsite.com/polymathreu}. The authors gratefully acknowledge the two referees whose comments have helped improve the exposition.

\vspace{1em}

\bibliographystyle{amsalpha}
\bibliography{references}


\end{document}